\documentclass[a4paper,11pt]{article}

\usepackage{amsfonts,amsmath,amsthm}
\usepackage[dvipsnames]{xcolor}
\usepackage{esint}
\usepackage[draft]{todonotes}
\usepackage{mathtools}

\usepackage{fancyhdr}
\usepackage[backend=biber, bibstyle=ieee, citestyle=numeric-comp, sorting=nyt,url=false]{biblatex}
\usepackage{hyperref}

\DeclareFieldFormat{prefixnumber}{\mkbibbold{#1}}
\DeclareFieldFormat{labelnumber}{\mkbibbold{#1}}

\definecolor{url}{HTML}{0000FF} 
\definecolor{bluesy}{HTML}{0057e7}
\definecolor{redsy}{HTML}{d62d20}
\definecolor{greensy}{HTML}{008744}

\hypersetup{
  linkcolor  = greensy,
  citecolor  = redsy,
  urlcolor   = bluesy,
  colorlinks = true,
}

\usepackage[capitalize,nameinlink]{cleveref}

\crefformat{equation}{Eq.\,(#2\textbf{#1}#3)}
\crefformat{theorem}{Theorem~#2\textbf{#1}#3}
\crefformat{lemma}{Lemma~#2\textbf{#1}#3}
\crefformat{proposition}{Proposition~#2\textbf{#1}#3}
\crefformat{remark}{Remark~#2\textbf{#1}#3}

\newcommand\R{\mathbb R}

\newcommand\Kcal{\mathcal K}
\newcommand\Lip{\mathop{\mathrm{Lip}}}
\newcommand\osc{\mathop{\mathrm{osc}}}
\newcommand\eps{\varepsilon}
\newtheorem{theorem}{Theorem}
\newtheorem{proposition}{Proposition}
\newtheorem{remark}{Remark}

\newtheorem{lemma}{Lemma}
\newtheorem{corollary}{Corollary}

\DeclareMathOperator{\loc}{loc}

\makeatletter
\renewcommand*\author[1]{%
  \stepcounter{author}%
  \ifnum\c@author=1
    \gdef\@author{#1}%
  \else
    \xdef\@author{\unexpanded\expandafter{\@author\and#1}}%
  \fi
  \csgdef{author@\the\c@author}{#1}}
\newcommand*\email[1]{%
  \csgdef{email@\the\c@author}{#1}}
\newcommand*\address[1]{%
  \csgdef{address@\the\c@author}{#1}}
\AtEndDocument{%
  \xdef\author@count{\the\c@author}%
  \c@author=1
  \print@authors}
\newcommand*\print@authors{%
  \ifnum\c@author>\author@count
  \else
    \print@author{\the\c@author}%
    \advance\c@author by 1
    \expandafter\print@authors
  \fi}
\newcommand*\print@author[1]{%
  \par\medskip
  \begin{tabular}{@{}l@{}}%
    \textsc{\csuse{author@#1}}\\
    \csuse{address@#1}\\
    E-mail address:
    \href{mailto:\csuse{email@#1}}{\csuse{email@#1}}
  \end{tabular}}
\makeatother

\title{Affine rigidity of functions with additive oscillation}
\author{Adolfo Arroyo-Rabasa}
\address{Dipartimento di Matematica\\
Universit\`a di Pisa\\ 
Largo Bruno Pontecorvo 5\\
56127 Pisa, Italy}
\email{adolfo.rabasa@unipi.it}

\author{Sergio Conti}
\address{Institut f\"ur Angewandte Mathematik\\
Universit\"at Bonn\\ 
Endenicher Allee 60\\
53115 Bonn, Germany}
\email{sergio.conti@uni-bonn.de}

\begin{document}
\maketitle

\begin{abstract}
We prove that a locally integrable function $f:(a,b) \to \R$ must be affine if its mean oscillation, considered as a function of intervals, can be extended to a locally finite Borel measure. 
In particular, we show that any function $f$  satisfying the integro-differential identity $|Df|(I)=4\osc(f,I)$ for {all} intervals $I \subset {(a,b)}$ must be affine.\\
 
\end{abstract}

\section{Introduction}

Let $a,b \in \R$ be two reals satisfying $a < b$ and let $f : (a,b) \subset \R \to \R$ be a locally integrable function. For an interval $I \subset
\subset (a,b)$, we define the mean value of $f$ over $I$ as
$$
(f)_I \coloneqq \frac{1}{|I|} \int_I f(x) \, dx = \fint_I f(x) \, dx.
$$
The oscillation of $f$ on $I$ measures the averaged deviation from its mean value:
$$
\osc(f, I) \coloneqq \fint_I |f - (f)_I| \, dx.
$$
Since $f$ is locally integrable, its distributional derivative $Df$ is well-defined as a linear functional on $C_c^\infty((a,b))$. The total variation of $Df$ on $I$ is given by the (extended) dual norm:
$$
|Df|(I) \coloneqq \sup \left\{ \int_I f \varphi' \, dx : \varphi \in C_c^\infty(I), \|\varphi\|_{L^\infty(I)} \leq 1 \right\} \in [0,\infty]\,.
$$
We say that $f$ has bounded  variation on $I$ if $|Df|(I) < \infty$. {We then denote such a function as $f \in BV(I)$, and if this holds for all $I \subset \subset (a,b)$, we say $f \in BV_{\text{loc}}((a,b))$.}

The Poincaré-Wirtinger inequality states that 
the oscillation of $f$ is controlled by its total variation:

\begin{equation*}
\osc(f,I) \leq \frac{1}{2} |Df|(I).
\end{equation*}
The constant $\frac 12$ is sharp, and equality holds precisely for $BV$ functions with a single jump discontinuity at the midpoint of $I$ (see~\cite{Hadwiger} and~\cite[Lemma~7.1]{ARBDN}). In this sense, one may think of the \emph{integro-differential equation}
\begin{equation*}\label{eq:12}
    \osc(f,I) = \frac 12 |Df|(I)
\end{equation*}
as one characterizing centered jump-type discontinuities on $I$.

In contrast, for affine functions, the Poincaré-Wirtinger quotient attains the smaller constant $\frac 14$ at every scale. Namely,

\begin{equation}\label{eq:affine}
\osc(f,I) = \frac{1}{4} |Df|(I) \quad \text{for all intervals } I \subset\subset (a,b).
\end{equation}

Motivated by the understanding of the fine properties of $BV$ functions through Poincar\'e quotients (see \cites{ambrosio2016bmo,arroyo2021representation,ARBDN}), Bonicatto, Del Nin, and the first author proposed in~\cite[Question~5]{ARBDN} that monotone functions satisfying~\cref{eq:affine} should be affine. The goal of this note is to give a positive answer to a more general version this question. Indeed, without assuming the monotonicity, we prove the following:  
\begin{theorem}\label{thm:1}
 Let $f : (a,b) \to \R$ be a locally integrable function satisfying~\cref{eq:affine}. Then, $f$ is affine. 
\end{theorem}
In fact, we prove a more general statement of the same result: 

 \begin{theorem}\label{thm:2}  Let $f : (a,b) \to \R$ be a locally integrable function and suppose that the function 
 \[
    \{\text{$I \subset \subset (a,b)$ open interval}\} \to \R : I \mapsto \osc(f,I)\,,
 \]
 extends to a positive Borel measure $\mu$ on $(a,b)$.
 
 Then, $\mu = \frac 14 |Df|$ and $f$ is affine.
 \end{theorem}

\section{Bulding blocks of the proof}
For $a,b \in \R$ with $a < b$ we write $\mathcal K_{a,b} \subset 2^{(a,b)}$ to denote the family of all open, nonempty and compactly contained sub-intervals of $(a,b)$. We write $|I|$ to denote the length of an interval $I\subset \R$.

\subsection{Invariance under affine transformations}

Let $f: (a, b) \to \mathbb{R}$ be locally integrable. Consider affine non constant transformations $L, F: \mathbb{R} \to \mathbb{R}$, {and suppose that} $F$ maps the interval $(a', b')$ onto $(a, b)$. By applying the area formula and the chain rule, we find that the function
$$\tilde{f}(x) = L f F(x), \quad x \in (a', b'),$$
satisfies
$$\osc(\tilde{f}, I) =
|L'| \cdot \osc(f, F(I))$$
and
$$|D\tilde{f}|(I) = |L'| \cdot |Df|(F(I))$$
for all $I \in \mathcal K_{a',b'}$.

Consequently, $f$ satisfies equation \cref{eq:affine} on $(a, b)$ if and only if $\tilde{f}$ satisfies \cref{eq:affine} on $(a', b')$. This property will be used repeatedly throughout the text.

\subsection{Lipschitz regularity of solutions}
The following lemma shows that the statements contained in~\cref{thm:1} and~\cref{thm:2} are equivalent, and that any function satisfying the assumptions of either must be locally Lipschitz.

 \begin{lemma}\label{prop:Lip_regularity} Let $f : (a,b) \to \R$ be locally integrable.  The following are equivalent:
 \begin{enumerate}
     \item\label{it:a} $f$ satisfies~\cref{eq:affine},
     \item\label{it:b} the function 
 \[
 \mathcal K_{a,b} \to \R : I \mapsto \osc(f,I)
 \]
 extends to a positive Borel measure $\mu \in \mathcal M_+((a,b))$.
 \end{enumerate}
Moreover, if $f$ satisfies the second property, then $f$ is locally Lipschitz. 
 \end{lemma}
\begin{remark}
    The only measure $\mu$ for which the second property may hold is $\mu = \frac 14 |Df|$. 
\end{remark}
\begin{proof} That $\ref{it:a} \Longrightarrow \ref{it:b}$ is a mere observation: Indeed, if $f$ satisfies~\cref{eq:affine}, then the fact that $f$ is locally integrable implies that $\frac 14 |Df|$ is a Radon measure extending the oscillation map over intervals.\\

To show that $\ref{it:b} \Longrightarrow \ref{it:a}$ and the conditional Lipschitz regularity, we divide the proof into steps:\\

\emph{Step 1. First, we show that $f \in BV_{\loc}((a,b))$ with $\frac14 |Df|\le \mu$.}
It suffices to show that
\begin{equation}\label{eqdfoscmu}
   \frac 14 |Df|(J) \le {\mu(J)} <  \infty\, \quad \text{for all $J \in \Kcal_{a,b}$.}
\end{equation}
Indeed, since $\mu$ is locally finite ($f$ is locally integrable) we deduce that $f \in BV_\mathrm{loc}((a,b))$ and since $|Df|$ is also a measure it follows that $\frac 14 |Df| \le \mu$ as measures.\\

The proof of \cref{eqdfoscmu} can be inferred from 
\cite[Th.~2.1]{FuscoMoscarielloSbordone2016} using convexity as in \cite[Prop. 5.1]{PonceSpector2017}. 
 For the reader's benefit, we give next a rigorous sketch of the proof.
 Assume first that $f\in C^1((a,b))$.
For any $x\in (a,b)$, one computes
\begin{equation*}
 \frac14 |f'|(x) = \lim_{\eps\to0}
 \frac{1}{\eps} \osc(f,(x-\eps/2,x+\eps/2)).
\end{equation*}
Since $J\subset\subset (a,b)$, the limit  is uniform for $x\in J$. Let  $\mathcal F_{J,\eps}$  denote a decomposition of $J$ into  pairwise disjoint intervals,
of length between $\eps$ and $2\eps$ (this exists, if $\eps$ is smaller than $|J|$).
 One checks
\begin{equation*}
\frac14 \int_J |f'|(x) dx \le \lim_{\eps\to0}
 \sum_{I\in \mathcal F_{J,\eps}} 
\osc(f,I)\le \mu(J).
\end{equation*}
For $\delta > 0$ small and $f\in L^1_{\loc}((a,b))$ we get $f_\delta\in C^\infty((J)_{\delta})$, {where $f_\delta$ is a standard $\delta$-scale mollification of $f$ and $(J)_\delta = \{x \in \R : \mathrm{dist}(x,J) < \delta\}$}.
 {Since $\mathcal F_{J,\eps}$ has finite cardinality, it follows that $g\mapsto\sum_{I \in \mathcal F_{J,\eps}} \osc (g,I)$ is a convex map from $L^1(J)$ to $\R$.} Hence, by Jensen's inequality  \begin{equation*}
\sum_{I\in \mathcal F_{J,\eps}} 
\osc({f_\delta|_J},I)\le 
\sup_{|y|<\delta} 
\sum_{I\in \mathcal F_{J,\eps}} 
\osc(f(\cdot +y),I)\le \mu((J)_\delta).
 \end{equation*}
Here, in passing to the last inequality we used that $\{I - y\}_{I \in \mathcal F_{J,\eps}}$ is a family of pairwise disjoint intervals.
 Letting $\eps\to0$,  for any $\delta>0$ sufficiently small, we obtain
\begin{equation*}
 \frac14 \int_J |f_\delta'|(x) \, dx
 \le \mu((J)_\delta).
\end{equation*}
{Subsequently letting  $\delta\to0$, we get $f_\delta \to f$ in $L^1(J)$ and by the $L^1$-lower semicontinuity of the total variation we discover that} 
\begin{equation*}
 \frac14 |Df|(J)\le\liminf_{\delta\to0}\frac14 \int_J |f_\delta'|(x) \, dx
 \le \limsup_{\delta\to0}\mu((J)_\delta)=\mu(\overline J) < \infty.
\end{equation*}
This proves~\cref{eqdfoscmu} and that $f \in BV(J)$ for any $J \in \Kcal_{a,b}$ with $\mu(\partial J)=0$. The case for general $J\in \Kcal_{a,b}$ follows from this and the fact that both $|Df|$ and $\mu$ are measures.   \\

\emph{Step 2. We observe that $\mu \le \frac 12 |Df|$ as measures.} Indeed, the Poincar\'e--Wirtinger inequality yields
\[
	\mu(I) = \osc(f,I) \le \frac 12 |Df|(I) \quad \text{for all $I \in \Kcal_{a,b}$\,.}
\]
The Borel regularity of $\mu$ yields that $\mu \le \frac 12 |Df|$ as measures.\\

\emph{Step 3. We prove that $f$ is locally Lipschitz.}
Since $f$ is a function with locally bounded variation, it follows from the embedding $BV(I) \to L^\infty(I)$ for $I \subset \subset \R$ that $f \in L^\infty_{\loc}((a,b))$. In particular,
\[
	M_I \coloneqq \|f\|_{L^\infty(I)} < \infty \quad \text{for all $I\in \Kcal_{a,b}$\,.}
\]

Now, fix $a < c <  d < b$ and let us write $I_x\coloneqq (c,x)$ for $x \in (c,d]$.
By the fundamental theorem of calculus, the function $r : (c,d] \to \R$ given by
\[
	r(x) \coloneqq \fint_{I_x} f(s)\,ds\,,
\] 
is absolutely continuous and differentiable almost everywhere with
\begin{equation}\label{eq:chain_rule_average}
\begin{split}
	\frac{d}{dx}r(x) & = \frac{d}{dx} \left( \frac{1}{x-c} \int_{c}^x f(s) \,ds\right) \\
	& = - \frac{1}{(x-c)^2} \int_{c}^x f(s)\,ds  + \frac{1}{x-c} f(x) = \frac{f(x) - r(x)}{|I_x|}\,.
\end{split}
\end{equation}
Moreover, $\sup_{x \in (c,d)} |r(x)| \le M_{(c,d)}$.
From \cref{eq:chain_rule_average} we deduce that for every $c'\in (c,d)$ the function $r$ is Lipschitz on $[c',d]$, with a Lipschitz constant that depends on $c$, $c'$, $d$.

Now, consider the function $h : (c,d) \times [-M_{I_d},M_{I_d}] \to \R$ given by
\[
	h(x,m) \coloneqq \int_c^x |f(s) - m| \, ds
\]
and notice that $h$ is Lipschitz, as the concatenation of Lipschitz maps. For almost every $(x,m) \in (c,d) \times [-M_{I_d},M_{I_d}]$ the fundamental theorem of calculus and the chain rule for Lipschitz maps give
\begin{equation}\label{eq:chain_rule_h}
\begin{split}
	\partial_x h(x,m) & = |f(x) - m| \le 2M_{I_d}\,,\\
	\partial_m h(x,m) & = \mathcal L^1(\{f < m\} \cap I_x) -  \mathcal L^1(\{f > m\} \cap I_x) \eqqcolon R(I_x,m)\,.
	\end{split}
\end{equation}
Since $R(I_{x},m) \le |I_{x}|$, it follows that $\mathrm{Lip}(h) \le 2 M_{I_{d}} + |I_d|$.

We define $\omega : (c,d] \to \R$
as the oscillation  
\[
	\omega(x) \coloneqq \osc(f,I_x)\,,
\]
and observe that 
\begin{equation}\label{eq:osc_lip}
	\osc(f,I_x) = \frac{1}{x -c} \cdot h\left(x, r(x)  \right)\,.
\end{equation}
Once again,
for any $c'\in (c,d)$
it follows from the chain rule that $\omega$
is Lipschitz on $[c',d]$, with Lipschitz constant $C_{c,c',d}$ depending solely on $c$, $c'$, and $d$. The estimate $|Df| \le 4\mu$ and the assumption yield \begin{align*}
	 \frac{|Df|((x-\eps,x+\eps))}{2\eps} & \le 
4\frac{\mu(x-\eps,x+\eps)}{2\eps}=
	 4  \frac{\osc(f,I_{x+\eps}) - \osc(f,I_{x-\eps})}{2\eps} \\
	 & = 4 \frac{\omega(x+\eps) - \omega(x - \eps)}{2\eps} \le 4 C_{c,c',d} < \infty\,,
\end{align*}
for all $x \in (c',d)$ and every $0 < \eps < \mathrm{dist}(x,\partial (c',d))$. This shows that
\[
\limsup_{r \to 0} \frac{|Df|((x-r,x+r))}{2r} \le 4 C_{c,c',d} < \infty \quad \text{for all $x \in (c',d)$.}
\]
Since the interval $(c,d) \subset \subset (a,b)$
and $c'\in (c,d)$ were chosen arbitrarily, the Lebesgue--Besicovitch  differentiation theorem yields the existence of $g \in L^1_{\loc}((a,b))$ with
\[
	\|g\|_{L^\infty(I)} \le 4C_I < \infty\quad \forall I \in \Kcal_{a,b}
\]
{and} such that
\[
	Df  = g \mathcal L^1 \quad \text{as measures on $(a,b)$.}
\] 
In particular, $f \in W^{1,\infty}_{\loc}((a,b))$ with $f' = g$. This proves that $f$ is locally Lipschitz on $(a,b)$.\\

\emph{Step 4. We show that $\mu = \frac 14 |Df|$.}
We know that $\mu\ll|Df|\ll\mathcal L^1$. If we can prove that
\begin{equation}\label{eqmudfx}
	\frac{\mu}{|Df|}(x) = \lim_{r \to 0} \frac{\mu ((x -r,x+r))}{|Df| ((x -r,x+r))} \le \frac 14
\end{equation}
for $\mathcal L^1$-almost every $x \in (a,b)$, the Besicovitch differentiation theorem and the Radon--Nykodym theorem yield the chain of inequalities
\[
	\frac 14 |Df| \le \mu = \frac{\mu}{|Df|}\, |Df| \le \frac{1}{4} |Df| \quad \text{as measures on $(a,b)$}\,.
\]
It follows that all inequalities must be equalities and hence $\mu = \frac 14 |Df|$.

It remains to prove \cref{eqmudfx}.  The proof can be directly obtained from
the results contained in \cite{ARBDN}, but we {give} a self-contained proof for {reader's benefit}:  Since $\mu \ll |Df| \ll \mathcal L^1$, we may assume that $x_0 \in \mathrm{spt}(|Df|)$ for otherwise there is nothing to show (this is used in the conclusion).
Let $x_0 \in (a,b)$ be a Lebesgue continuity point of $f'$. Shortening $I_r\coloneqq(x_0-r,x_0+r)$ and
letting $g(x) \coloneqq f(x) - f(x_0) - f'(x_0)(x - x_0)$, we deduce from Poincar\'e's inequality that
\[
     \frac 1{2r} \fint_{I_r} |g - (g)_{I_r}| \, dx \le \frac 12 \fint_{I_r} |f' - f'(x_0)| \, dx \to 0 \quad \text{as $r \to 0$}.
\]
Since $g - (g)_{I_r} = f - (f)_{I_r} - f'(x_0)(x-x_0)$, it follows  that
\[
   \limsup_{r \to 0} \frac 1{2r} \fint_{I_r} |f - (f)_{I_r} - f'(x_0)(x-x_0)|  \, dx = 0.
\]
Therefore,
\[\begin{split}\lim_{r\to0}
\frac{\mu(I_r)}{2r}=&
\lim_{r\to0}
\frac1{2r}
\fint_{I_r}
|f(x)-(f)_{I_r}| \, dx\\
=& \lim_{r \to 0} \frac 1{2r} \fint_{I_r} |f'(x_0)(x - x_0)| \, dx \\
=&
 |f'(x_0)| \lim_{r\to0}
 \frac{1}{4r^2}
\int_{I_r}
|x-x_0| dx
=\frac14 |f'(x_0)|.
\end{split}
\]
Since almost every $x_0 \in (a,b) \cap \mathrm{spt}(|Df|)$ is a Lebesgue continuity point of $f'$ and at such points it holds $\lim_{r\to0} 2r |f'(x_0)| /|Df|(I_r) =1$, this concludes the proof of~\cref{eqmudfx}.

This finishes the proof.
\end{proof}

\subsection{Necessary Conditions}
Based on the preceding observations, we assume that $f: (a, b) \to \mathbb{R}$
satisfies \cref{eq:affine} and is locally Lipschitz with a weak derivative $f' \in L^\infty_{\mathrm{loc}}((a, b))$.
We also recall from~\cref{eq:osc_lip} that, {for any given $(a',b') \in \Kcal_{a,b}$,}
\[
     {(a',b')} \to \R: x \mapsto \osc(f,({a'},x))
\]
is Lipschitz {(with Lipschitz constant depending on $a',b'$)}. 
For $a < x < y < b$, we write
\[
m_{x,y} \coloneqq (f)_{(x,y)} = \frac{1}{y-x}\int_x^y f(t) \, dt\,.
\]
We also define
\begin{equation*}
 R(I)\coloneqq
 |\{f < (f)_{I}\} \cap I| - |\{f > (f)_{I}\} \cap I|\,, \qquad I \in \Kcal_{a,b}\,.
\end{equation*}

\subsubsection{One-Sided Variations on Intervals}
\begin{lemma}\label{lemmaode}
 If $f : (a,b) \to \R$ is locally integrable and satisfies~\cref{eq:affine}, then:
 \begin{enumerate}
\item\label{lemmaodeode}
For any $x\in (a,b)$ and almost any $y\in (x,b)$ we have
\begin{equation}\label{eqodexyy}
\begin{split}
\frac14 |f'(y)| =&
 -\frac{\osc(f,(x,y))}{y-x}
 + \frac{|f(y)-m_{x,y}|}{y-x}
 +
 \frac{R((x,y))}{y-x}
 \cdot
 \frac{f(y)-m_{x,y}}{y-x},
 \end{split}
\end{equation}
similarly for any $y\in (a,b)$ and almost any $x\in (a,y)$ we have
\begin{equation}\label{eqodexyx}
\begin{split}
\frac14 |f'(x)| =&
 -\frac{\osc(f,(x,y))}{y-x}
 + \frac{|f(x)-m_{x,y}|}{y-x}
+
 \frac{R((x,y))}{y-x}
 \cdot
 \frac{f(x)-m_{x,y}}{y-x}
 \end{split}
\end{equation}
(equivalently, both hold for $\mathcal L^2$-almost any $(x,y)$ such that $a<x<y<b$).
\item\label{lemmaoderem:1}
If $y$ is a {(Lebesgue)}  continuity point of
$y\mapsto R((x,y))$
then $y$ is a {(Lebesgue)} continuity point of $|f'|$. 
\item\label{lemmaoderem:2}
For any $x\in (a,b)$ and any $y\in (x,b)$ we have 
\[
\osc(f,(x,y))\le 2 |f(y)-m_{x,y}|.
\]
 \end{enumerate}
\end{lemma}
\begin{proof}
The argument used to prove~\ref{lemmaodeode}
is similar to the one in Step 3 of the proof of Lemma~\ref{prop:Lip_regularity}.
We pick $x$, $y$ such that $a<x<y<b$, and
 differentiate the equation
 $\frac14 |Df|((x,y)) = \osc(f, (x,y))$ in $y$.
 For any fixed $x\in (a,b)$, differentiating in $y$ we obtain
 \begin{equation*}
\frac{d}{dy}\frac14 |Df|((x,y))=\frac{|f'|(y)}4
 \end{equation*}
 for almost any $y\in (x,b)$.
In turn, the computation in \cref{eq:chain_rule_average} {and the continuity of $f$} give
\begin{equation}\label{eqdymxy}
\frac{d}{dy}m_{x,y}=\frac{f(y)-m_{x,y}}{y-x}
\end{equation}
{for all $y \in (x,b)$.} 
We let
\begin{equation*}
 h(x,y,m)\coloneqq\int_x^y |f(s)-m|\, ds
\end{equation*}
and observe that it is locally Lipschitz,
with
$\partial_y h(x,y,m) =|f(y)-m|$.
Further, for any $m\in\R$ such that $|\{f=m\}\cap (x,y)|=0$, we also have
\begin{equation*}
\partial_m h(x,y,m)=
|\{f < m\} \cap (x,y)| - |\{f > m\} \cap (x,y)|.
\end{equation*}
By the chain rule for Lipschitz functions,
for any fixed $x$ and any $y\in (x,b)$,
{with} $|\{f=m_{x,y}\}\cap (x,y)|=0$, it holds
\begin{equation}\label{eq:derivative_osc}\begin{split}
 \frac{d}{dy} \osc(f,(x,y))=&
 \frac{d}{dy} \frac{h(x,y,m_{x,y})}{y-x}\\
 =&
 -\frac{\osc(f,(x,y))}{y-x}
 + \frac{|f(y)-m_{x,y}|}{y-x} \\
 & \qquad +
 \frac{R((x,y))}{y-x}
 \cdot
 \frac{f(y)-m_{x,y}}{y-x}.
 \end{split}
\end{equation}

We will show below that
\begin{equation}\label{eqmyzero}
\partial_y m_{x,y}=0 \text{ for a.e. } y\in (x,b) \text{ such that } |\{f=m_{x,y}\}\cap (x,y)|>0.
\end{equation}
The  vanishing of $\partial_y m_{x,y}$ has the following implications. First, together with~\cref{eqdymxy}, it implies $f(y)=m_{x,y}$. As a result, the last two terms in~\cref{eq:derivative_osc} vanish. In particular,~\cref{eq:derivative_osc} holds  provided we show that $\partial_y h(x,y,m_{x,y}) = 0$. We verify that the latter is also a consequence of $\partial_y m_{x,y} = 0$. Indeed, on the one hand, 
\begin{align*}
    |h(x,y,m_{x,y}) - h(x,y,m_{x,y'})| & \le \Lip(h(x,y,\cdot)) |m_{x,y} - m_{x,y'}|\\
    & = \mathrm{o}_{x,y}(|y-y'|).
\end{align*}
On the other hand, the fundamental theorem of calculus yields 
\begin{align*}
    h(x,y,m_{x,y'}) - h(x,y',m_{x,y'}) & = \int_y^{y'} \partial_{t} h(x,t,m_{x,y'}) \, dt  \\
    & = \int_y^{y'} |f(t) - m_{x,y'}| \, dt \\
    & \le \int_y^{y'} |f(t) - m_{x,y}| +\mathrm{o}(|y-y'|) \, dt.
\end{align*}
Since $f$ is continuous with $f(y) = m_{x,y}$, this entails 
\[
    |h(x,y,m_{x,y'}) - h(x,y',m_{x,y'})| = \mathrm{o}_{x,y}(|y - y'|).
\]
Gathering these two estimates we deduce
\[
    \frac{d}{dy} m_{x,y} = 0 \quad \Rightarrow \quad \frac{d}{dy} h(x,y,m_{x,y}) = 0.
\]
We conclude that~\cref{eqmyzero} conveys the validity of~\cref{eq:derivative_osc} for a.e. $y \in (x,b)$.

In order to prove \cref{eqmyzero}, we observe that for any $\alpha\in\R$ one has $\partial_y m_{x,y}=0$ almost everywhere on $\{y: m_{x,y}=\alpha\}$. At the same time, there are at most countably many $\alpha\in\R$ such that $| (a,b)\cap \{ f=\alpha\}|>0$.
{Since the countable union of negligible sets is negligible, we conclude that} $\partial_y m_{x,y}=0$ almost everywhere on $\{y: |\{f=m_{x,y}\}\cap (x,y)|>0\}$. 
This concludes the proof of \cref{eqmyzero}
and therefore of
\cref{eqodexyy}.

The argument for the other endpoint is identical,
using $\partial_x h(x,y,m)=-|f(x)-m|$ and $\partial_x m_{x,y}=(m_{x,y}-f(x))/(y-x)$ leads to
\cref{eqodexyx}
(alternatively, one applies \cref{eqodexyy} to $g(s)\coloneqq-f(b+a-s)$).

To prove \ref{lemmaoderem:1}, we observe that
all terms on the right-hand side of \cref{eqodexyy}, except possibly for $y \mapsto R((x,y))$, are continuous on $(x,b)$. Therefore, a point $y \in (x,b)$  is a (Lebesgue) continuity point of $|f'|$
if and only if it is a (Lebesgue) continuity point of $R((x,y))$. 
The  sought assertion then follows from the fact that $f$ is Lipschitz.

To prove \ref{lemmaoderem:2}, we observe that
by definition $|R(I)|\le |I|$, so the conclusion for a.e. $y \in (x,b)$ follows immediately from the fact that the left-hand side of \cref{eqodexyy} is nonnegative. The assertion for all $y \in (x,b)$ follows from the fact that both sides of the sought inequality are continuous on $(x,b)$ as a function of $y$.

This finishes the proof.
\end{proof}

\subsubsection{Monotonicity}

\begin{lemma}\label{lem:extremal} If $f : (a,b) \to \R$ is locally integrable and satisfies~\cref{eq:affine}, then 
one (and only one) of the  following conditions holds for any 
 $a < x < y < b$: 
    \begin{enumerate}
        \item
        $f(y)\ne m_{x,y}$,
                \label{item:2}
        \item $f$ is a constant on $(x,y)$. 
    \end{enumerate}
\end{lemma}
\begin{proof} Suppose that $f(y)=m_{x,y}$. By Lemma~\ref{lemmaode}, 
we have $\osc(f,(x,y))=0$, which can only occur if $f$ is constant on $(x,y)$. 
\end{proof}
\begin{corollary}[Monotonicity]
    If $f(x) = f(y)$ for some $x,y \in (a,b)$ with $x <y$, then $f$ is constant on $(x,y)$. In particular, $f$ is monotone.  
\end{corollary}
    \begin{proof} We argue through a contradiction argument by assuming that $f$ is non-constant on $(x,y)$.  In this situation it holds $\min_{[x,y]} f < m_{x,y} < \max_{[x,y]} f$. 
    By Lemma~\ref{lem:extremal}, $f(y)\ne m_{x,y}$.
    Up to multiplication by a sign, we may assume that 
    \begin{equation}\label{eq:max_contra}
    f(y) = f(x) < m_{x,y} \le M \coloneqq \max_{[x,y]} f.
    \end{equation}
    We may find $x_0 \in (x,y)$ satisfying $f(x_0) = M$. 
    Let $r(t) \coloneqq m_{x,t}$, so that
    $r(x_0) \le f(x_0)$, and $r(y) > f(y)$.
    By the mean value theorem, we may find $y_0 \in [x_0,y)$ satisfying $f(y_0) = r(y_0)$. Lemma~\ref{lem:extremal} implies that $f$ is constant on $(x,y_0)$ and therefore $f(x) = f(x_0) = M$, contradicting~\cref{eq:max_contra}.
    \end{proof}

    \section{Regularity of solutions}

    \subsection{$BV$-regularity of the derivative} An immediate consequence of the monotonicity and the optimality equations for $|f'|$, is the fact that $f'$ is a function with locally bounded variation.
    
    \begin{proposition} If $f : (a,b) \to \R$ is locally integrable and satisfies~\cref{eq:affine}, then
         $f' \in BV_{\loc}((a,b))$. In particular, the left and right (classical) derivatives of $f$ exist at every $x \in (a,b)$. 
    \end{proposition}

\begin{proof}
Due to monotonicity, we may assume without loss of generality that $f$ is non-decreasing and hence that $f' \ge 0$ almost everywhere.
It suffices to show that $f'\in BV_{\loc}((x,b))$ for any $x\in (a,b)$.
We use \cref{eqodexyy} and prove that each term is in $BV_{\loc}((x,b))$. First, $y\mapsto \osc(f,(x,y))$ is locally Lipschitz, and hence in $BV_{\loc}$. The functions $y\mapsto f(y)$ and
$y\mapsto m_{x,y}$ are non-decreasing and hence in $BV_{\loc}$. Further,
\begin{equation*}
  R((x,y))=
 |(x,y)\cap \{f<m_{x,y}\}|
  + |(x,y)\cap \{f\le m_{x,y}\}|-(y-x),
\end{equation*}
and since the first two terms are non-decreasing in $y$, also $R((x,\cdot))\in  BV_{\loc}$. {Since $f' = |f'|$ is the sum of these three terms, we obtain $f' \in BV_{\loc}((x,b))$, as desired.}

We recall that functions of bounded variation of one variable possess left and right semicontinuous representatives (see~\cite[Theorem 3.28]{APF_Book}). We can use this show that the limits
\[
    f'(x_-) \coloneqq \lim_{y \to x_-} \frac{f(x) - f(y)}{x-y} \quad \text{and} \quad f'(x_+)  \coloneqq \lim_{y \to x_+} \frac{f(x) - f(y)}{x-y}
\]
exist for every $x \in (a,b)$.

Indeed, let $g$ be the left continuous representative of $f'$ and observe that 
\[
\lim_{\eps\to0}    \frac{f(x) - f(x - \eps)}{\eps} =\lim_{\eps\to0}   \fint_{x-\eps}^x f'(s) \, ds =  g(x) \,.
\]
We conclude that $f'(x_-)$ exists and equals $g(x)$. 
The proof for $f'(x_+)$ is analogous. 
\end{proof}

    \subsection{$C^1$ regularity of solutions}
    We now use the existence of one-sided derivatives to prove that $f$ is differentiable everywhere. In fact, we prove that $f$ is continuously differentiable.

    \begin{proposition}\label{prop:C1}
        If $f : (a,b) \to \R$ is locally integrable and  satisfies~\cref{eq:affine}, then $f$ is continuously differentiable on $(a,b)$. 
    \end{proposition}
    \begin{proof}
        Let $x \in (a,b)$. Without loss of generality we may assume that 
    \[
    \text{$(-1,1) \subset (a,b)$, \quad $f' \ge 0$, \quad $x = 0$ \quad and \quad $f(0) = 0$.}
    \]
    The existence of one-sided derivatives implies that 
    \begin{equation}\label{eqfapproxpwaff}
        f(s) = \chi_{(-1,0)} \alpha s +  \chi_{(0,1)}\beta s + \mathrm{o}(s)\,,\qquad  s \in (-1,1)\,,
    \end{equation}
        where $\alpha \coloneqq f'(0_-)$, $\beta \coloneqq f'(0_+)$ are the left and right derivatives of $f$ at $x = 0$.
    Our goal is to show that $f$ is differentiable, i.e., $\alpha = \beta$.
    
 We shorten
 $m_\eps\coloneqq m_{-\eps,\eps}$. Since $f$ is nondecreasing, we obtain $f(-\eps) \le m_{\eps} \le f(\eps)$. Further,  integrating \cref{eqfapproxpwaff}, 
    \begin{equation}\label{eq:average_eps}
4m_{\eps}   = f(-\eps) + f(\eps) + \mathrm{o}(\eps) \text{ for every $\eps \in (0,1)$.}
    \end{equation}
    Re-writing~\cref{eqodexyy} and~\cref{eqodexyx} with  $I_\eps \coloneqq (-\eps,\eps)$ yields, for almost every $\eps \in (0,1)$, the equations
    \begin{align*}
        \frac 14 f'(\eps) & = - \frac{\osc (f,(-\eps,\eps))}{2\eps} + \frac{f(\eps) - m_{\eps}}{2\eps} + \frac{f(\eps) - m_{\eps}}{2\eps}\cdot \frac{R(I_{\eps})}{2\eps}\,, \\
        \frac 14 f'(-\eps) & = - \frac{\osc (f,(-\eps,\eps))}{2\eps} + \frac{m_{\eps} - f(-\eps)}{2\eps} + \frac{f(-\eps)-m_\eps }{2\eps} \cdot \frac{R(I_{\eps})}{2\eps}\,.
    \end{align*}
    Subtracting the second from the first one, we deduce that
    \begin{equation}\label{eqC1}
 \frac{f'(\eps) - f'(-\eps)}4  =
 \frac{f(\eps)+f(-\eps)-2m_\eps}{2\eps}
 + \frac{f(\eps)-f(-\eps)}{2\eps}
\cdot \frac{R(I_\eps)}{2\eps}
    \end{equation}
    for almost every $\eps \in (0,1)$. By
    \cref{eq:average_eps}, the first term  {on the RHS}
   equals $\frac{m_\eps}{\eps}+\mathrm{o}(1)$. In order to evaluate the second term, we need to compute $R(I_\eps)$.
    We first recall, {from} \cref{eqfapproxpwaff} and \cref{eq:average_eps}, {that}
    \begin{equation*}
    m_{\eps}  = \frac{f'(\eps) - f'(-\eps)}4 + \mathrm{o}(\eps)=\frac{\beta-\alpha}4 \eps + \mathrm{o}(\eps).
    \end{equation*}
If $\alpha=\beta$, we are done. If not, $\eps\mapsto m_\eps$ is, for small $\eps$, injective, therefore for almost every {(sufficiently small)} $\eps$
the set $\{f=m_\eps\}$ is a null set.  By monotonicity of $f$, this set is an interval, therefore for almost every $\eps$ it contains at most a single point.
    By continuity of $f$, there is a unique $t_\eps\in I_\eps$ such that
    $f(t_\eps)=m_\eps$.
    In particular, $R(I_\eps)=2t_\eps$.
    If $\alpha<\beta$ then for small $\eps$ we have
    $m_\eps>0$, so that $t_\eps>0$, and
     inserting in \cref{eqfapproxpwaff}
    leads to
    \begin{equation*}
     t_\eps=\frac{\beta-\alpha}{4\beta}\eps+\mathrm{o}(\eps)
    \end{equation*}
for almost all $\eps$.
    {Inserting this into \cref{eqC1} gives us:
\begin{equation*}
  \frac{\beta-\alpha}{4} = \frac{\beta-\alpha}{4}+\frac{\beta+\alpha}{2} \frac{t_\eps}{\eps} + \mathrm{o}(1).
\end{equation*}
Rearranging terms and inserting the value of $t_\eps$ this becomes
\begin{equation*}
0=\frac{\beta+\alpha}{2}\cdot \frac{\beta-\alpha}{4\beta} + \mathrm{o}(1),
\end{equation*}
which, since $\alpha$ and $\beta$ are nonnegative, implies $\alpha=\beta$.
If $\beta<\alpha$ instead, a similar calculation yields:
\begin{equation*}
  t_\eps=\frac{\beta-\alpha}{4\alpha}\eps+\mathrm{o}(\eps),
\end{equation*}
and the conclusion follows in the same manner.}

    This argument proves that $f$ is differentiable everywhere on $(a,b)$. In particular, 
    the fact that $f'(x_-)=f'(x_+)$
    shows that $Df'$ is non-atomic. By the classical theory $BV$ functions of one variable (see, e.g.,~\cite[Theorem~3.28]{APF_Book}), it follows that $f'$ is continuous on $(a,b)$ and hence $f$ is continuously differentiable  there too.
    \end{proof}

    \subsection{Partial smoothness of solutions}
    Once we have shown solutions are continuously differentiable, we can easily bootstrap their regularity through the optimality condition in~\cref{eqodexyy}. In particular, we show next that solutions of~\cref{eq:affine} are smooth:

    \begin{proposition}[Partial smoothness]\label{prop:partial_smoothness} Let $f : (a,b) \to \R$ be a locally integrable function satisfying~\cref{eq:affine}.
         Suppose that $f'(x_0) \neq 0$ for some $x_0 \in (a,b)$. Then, $f$ is smooth in an open neighborhood of $x_0$.  
    \end{proposition}
      \begin{proof} Recalling that $f$ is of class $C^1$, we may assume $f ' > 0$ on $I = (a',b') = (x_0-\eps,x_0 +\eps) \subset \subset (a,b)$ for some $\eps = \eps(x_0)>0$.  
            
       Let $k \ge 0$ be an arbitrary nonnegative integer. We proceed by induction on the order of differentiability: it suffices to show that if $f \in C^{k}(I)$, then $f \in C^{k+1}(I)$. The base case, $k=1$, follows from \cref{prop:C1}.  

      Assume that $f \in C^{k}(I)$. By the inverse function theorem, $f$ is invertible on $I$ with a $C^k$ inverse $f^{-1}$.   Casting into~\cref{eqodexyy} the fact that $f(x) > m_{a',x} $ for every $x \in I$, we can express $f'(x)$  as
      \[
            \frac {f'(x)}{4} = - \frac{\osc(f,(a',x))}{x -a'} + \frac{f(x) - m_{a',x}}{x - a'} + \frac{f(x) - m_{a',x}}{x - a'} \cdot\frac{R((a',x))}{x - a'}\,.
      \]
Both sides are continuous, therefore they are equal everywhere.
Let us analyze the regularity of the terms in the right-hand side. Since $f$ is of class $C^k$, it follows that the second term is of class $C^k$ as well. On the other hand, the fact that $f$ is invertible on $I$ implies that the third term is also of class $C^k$. Indeed,  since
      \begin{align*}
          |\{f < m_{a',x} \}\cap (a',x)| & = f^{-1}(m_{a',x}) - a'\,,\\
            |\{f > m_{a',x} \}\cap (a',x)| & = x - f^{-1}(m_{a',x})\,,
      \end{align*}
      it follows that $R((a',x))$ can be written as sum of an affine transformation and  $x \mapsto 2f^{-1} \circ m_{a',x}$; the latter being a $C^k$ map.

To deal with the regularity of the first term, we write
\begin{equation*}\begin{split}
&(x-a') \osc(f,(a',x))=
\int_{a'}^x |f(s)-m_{a',x}| ds\\
&\hskip1cm=
\int_{a'}^{f^{-1}(m_{a',x})} (f(s)-m_{a',x}) ds+
\int_{f^{-1}(m_{a',x})}^x (m_{a',x}-f(s)) ds,
\end{split}\end{equation*}
which by the same argument is a $C^k$ map.

      This shows the right-hand side is of class $C^k$ on the interval $I$, and hence $f\in C^{k+1}(I)$. This finishes the proof. \end{proof}

\section{Proof of the main result}
    \subsection{Affine rigidity under smoothness}
    Here, we show that~\cref{thm:1} holds under the additional assumption that $f$ is smooth and strictly monotone: 
    
    \begin{proposition}\label{prop:Taylor}
  If $f : (a,b) \to \R$ is a smooth function satisfying~\cref{eq:affine} and {$f' \neq 0$ in $(a,b)$}, then $f$ is affine.
 \end{proposition}
\begin{proof} Let $k$ be a positive integer, which we shall fix later in the proof. 

First, let $x_0 \in (a,b)$. In the following we let $\eps>0$ be a small positive real and consider the interval $I_\eps\coloneqq(x_0-\eps, x_0+\eps) \subset\subset (a,b)$. Since~\cref{eq:affine} is invariant under the composition {with affine transformations, we may henceforth assume that $f(x_0) = 0$ and $f' > 0$.} Provided that $\eps$ is sufficiently small, the $k$th order Taylor series expansion of $f$ exists on $I_\eps$ and satisfies 
 \begin{equation*}
  f(x_0+t)=\sum_{j=1}^k A_j t^j + \mathrm O(t^{k+1})\,, \qquad t \in (-\eps,\eps)\,, A_1 \ne 0\,.
 \end{equation*}
Computing the average as a function of $\eps$ we obtain 
 \begin{equation*}
  (f)_{I_\eps} = g(\eps) \coloneqq \sum_{j=2}^k\frac{A_j}{j+1}  \eps^j \chi^{\mathrm{even}}(j) 
  + \mathrm O(\eps^{k+1}) \,,
 \end{equation*}
 where $\chi^{\mathrm{even}}: \mathbb Z \to \{0,1\}$ is the indicator function of the even integers $2\mathbb Z$. 
 Here, the odd powers on the right-hand side disappear because the average on $I_\eps$ of any odd function is zero.
By the strict monotonicity of $f$ on $I_\eps$ for sufficiently small $\eps$, the equation $f(x_0 + s)= (f)_{I_\eps}$ has a unique solution  for $s \in (-\eps,\eps)$, which we shall denote by $\rho$. It obeys
 \begin{equation}\label{eqrhofroimeps}
  \sum_{j=1}^k A_j \rho^j + \mathrm O(\rho^{k+1})
  =\sum_{j=2}^{k} \frac{A_j}{j+1}  \eps^j \chi^\mathrm{even}(j)
  +\mathrm O(\eps^{k+1}).
 \end{equation}
 Therefore $\rho$ is of order $\eps^2$. Indeed, letting $k\ge 3$ we obtain
 \begin{equation}\label{eqrhok2}
 \rho = \frac{A_2}{3A_1}\eps^2  + \mathrm O(\eps^4).
\end{equation}
Notice that $I_\eps\cap \{f>(f)_{I_\eps}\}=(x_0+\rho,x_0+\eps)$. Therefore, casting the expressions for $f$ and $(f)_{I_\eps}$ above
and using that for any interval $I$
\begin{equation*}
\int_I |f-(f)_I|\, dx=2\int_{I\cap\{f>(f)_I\}} (f-(f)_I) \, dx
=2\int_{I\cap\{f<(f)_I\}} ((f)_I -f)\, dx
\end{equation*}
yields
\begin{equation*}\begin{split}
 \int_{I_\eps} |f-(f)_{I_\eps}| \, dx & = 2\int_{x_0+\rho}^{x_0+\eps} (f-(f)_{I_\eps}) \, dx \\
 & = 2\int_{x_0+\rho}^{x_0+\eps} f \, dx  
 -2(\eps-\rho) (f)_{I_\eps}\,.
\end{split}
\end{equation*}
The oscillation $h_\eps \coloneqq \osc(f,I_\eps)$ can therefore be expressed as
\begin{equation*}\begin{split}
h_\eps = \sum_{j=1}^k \frac{A_j}{j+1} (\eps^j-\frac{\rho^{j+1}}{\eps})
 -(1-\frac\rho\eps)  \sum_{j=2}^{k} \frac{A_j}{j+1}  \eps^j \chi^\mathrm{even}(j)
 +\mathrm O(\eps^{k+1}) .
\end{split}\end{equation*}
Inserting~\cref{eqrhofroimeps} into this expression we get
\begin{equation*}
\begin{split}
h_\eps = & \sum_{j=1}^k \frac{A_j}{j+1} \eps^j
\chi^\mathrm{odd}(j)
-
\frac \rho \eps \sum_{j=1}^k  \left(\frac{A_j}{j+1} \rho^{j}
-   \frac{A_j}{j+1}  \eps^j \chi^\mathrm{even}(j)\right)+\mathrm O(\eps^{k+1})\\ 
     = & \sum_{j=1}^k \frac{A_j}{j+1} \eps^j
\chi^\mathrm{odd}(j)
- \frac \rho \eps \sum_{j=1}^k \left( \frac{A_j}{j+1} \rho^{j}
-  A_j  \rho^j \right)
 +\mathrm O(\eps^{k+1}) \\
    = & \sum_{j=1}^k \frac{A_j}{j+1} \eps^j
\chi^\mathrm{odd}(j)
+
\frac \rho \eps \sum_{j=1}^k \frac{j A_j}{j+1} \rho^{j}
 +\mathrm O(\eps^{k+1})\,,
\end{split}\end{equation*}
where $\chi^\mathrm{odd} = 1 - \chi^\mathrm{even}$. At this point we choose $k=4$, and use the expansion for $\rho$ in \cref{eqrhok2} to get
\begin{equation*}\begin{split}
h_\eps =&
\frac{A_1}{2}\eps + 
\left(\frac{A_3}{4}
+\frac{A_2^2}{18A_1}\right)\eps^3
+\mathrm O(\eps^5)\,,
\end{split}\end{equation*}
which {(using the monotonicity of $f$ on $I_\eps$)} by assumption is to be compared with 
\begin{equation*}
 \frac14 |Df|(I_\eps)=
 \frac{{f(\eps)-f(-\eps)}}{4}
 =
 \frac12 \sum_{j=1}^k A_j\eps^j \chi^\mathrm{odd}(j)+\mathrm O(\eps^{k+1})\,.
\end{equation*}
Therefore, for $k=4$ we obtain
\begin{equation*}\begin{split}
&
\frac{A_1}{2}\eps +
\left(\frac{A_3}{4}
+\frac{A_2^2}{18A_1}\right)\eps^3
=\frac{A_1}2 \eps + \frac{A_3}2  \eps^3 + \mathrm O(\eps^5)\,.
\end{split}\end{equation*}
Since this holds for every sufficiently small $\eps$, we deduce that $9A_3A_1=2A_2^2$. Inserting the Taylor coefficient rule $A_j=f^{(j)}(x_0)/j!$, this translates into
\begin{equation*}
3 f'(x_0) f'''(x_0)=(f''(x_0))^2\,.
\end{equation*}

It follows that $f'$ satisfies the ordinary differential equation
\begin{equation*}\label{eq:g}
 3g(t) g''(t)=(g'(t))^2 \qquad \text{for all $t \in (a,b)$.}  
\end{equation*}
This equation has two types of solutions: constants and functions of the form
\begin{equation*}
 g(x)=A(x-B)^{3/2}
\end{equation*}
for any $A, B\in\R$. From this, {monotonicity} and smoothness we deduce that $f$ must either be affine or of the form
\[
    f(x) = A(x - B)^{\frac 52}+C, \qquad A \neq 0, B \notin {(a,b)}, C\in\R\,.
\]
We claim that $f$ cannot be of the latter form. Indeed, due to the invariance of~\cref{eq:affine} under affine transformations on the domain and co-domain, the claim is equivalent to observing that $x^{\frac 52}$ is not a solution of~\cref{eq:affine} on any interval  
 $(\alpha,\beta)$ with $0 < \alpha < \beta$. 
 We show that if the function $f(x)=x^s$ solves~\cref{eq:affine} on some interval $(\alpha,\beta)$ with $0 <\alpha<\beta$ then $s$ cannot be $5/2$. We can assume $s>0$. Our assumption means that the function 
\begin{equation}\begin{split}
 \varphi(a,b)\coloneqq&\frac{f(b)-f(a)}{4}-\frac{2}{b-a}
 \int_a^{x(a,b)}
 f(x(a,b))-f(z) \, dz, \text{ where}\\ x(a,b)\coloneqq&\left(\frac{1}{b-a} 
 \int_a^b f(z) \, dz\right)^{1/s}, 
\end{split}\end{equation}
vanishes
for any $a$, $b$ with $0<\alpha<a<b<\beta$. {Since the integral of a real-analytic function is also real-analytic,} it follows that both $x$ and $\varphi$ are real-analytic on $0<a<b$. {We} obtain $\varphi=0$ on $0<a<b$, {and, by continuity,} also on $0\le a<b$. 
Computing
\begin{equation}\begin{split}
x(0,b)=&\frac{b}{(1+s)^{1/s}}, \\
 \varphi(0,b)=&\frac14 b^s
 -\frac2b \left[x(0,b)^{1+s}-\frac{x(0,b)^{1+s}}{1+s}\right]
 =b^s\left[\frac14 -\frac{2s}{(1+s)^{2+1/s}}\right],
\end{split}\end{equation}
 one checks that 
$8s=(1+s)^{2+1/s}$ is not solved by $s=5/2$ (but {it} is, of course, solved by $s=1$).

We conclude that $f$ must be affine.
\end{proof}

    Endowed with the full $C^1$-regularity of solutions (\cref{prop:C1}), the partial smoothness of solutions on the set of points with non-vanishing derivative (\cref{prop:partial_smoothness}) and the validity of~\cref{thm:1} for smooth {maps with non-zero derivative} (\cref{prop:Taylor}), we can now prove the main result:\\

\noindent \textbf{Theorem 1.}\emph{
 Let $f : (a,b) \to \R$ be a locally integrable function satisfying~\cref{eq:affine}. Then, $f$ is affine. 
}

\begin{proof} Let $U \subset (a, b)$ be the set of points where $f'$ is nonzero. Since $f$ is $C^1$ on $(a, b)$, $U$ is an open set. We distinguish two cases: a) If $U$ is empty, then $f' \equiv 0$ on $(a, b)$, and hence $f$ is constant; b) If $U$ is non-empty, then it is the union of disjoint open intervals. Let $I$ be one such interval. By~\cref{prop:Taylor}, $f$ is affine on $I$. Since $f'$ is continuous (\cref{prop:C1}) on $(a,b)$, and constant and non-zero on $I$, it follows that $f' \neq 0$ on $\partial I$ and hence $I$ must be the entire interval $(a, b)$. In either case, $f$ is affine. \end{proof}

The proof of~\cref{thm:2} follows directly from~\cref{thm:1} and~\cref{prop:Lip_regularity}.

\subsection*{Acknowledgments} 
This work was partially supported
by the Deutsche Forschungsgemeinschaft (DFG) through project SFB1060 - 211504053
and project CRC1720 - 539309657.
AAR was also supported by the European Union European Research Council through the ERC Starting Grant ConFine
(no. 101078057).

\printbibliography

\end{document}